\documentclass[12pt]{amsart}

\usepackage{fullpage}
\usepackage{amsmath}
\usepackage{amssymb}
\usepackage{amsthm}
\usepackage{amstext}
\usepackage{appendix}
\usepackage{perpage}
\usepackage[dvipsnames]{xcolor}
\usepackage{tikz-cd}
\usepackage[colorlinks=true, allcolors=MidnightBlue]{hyperref}

\newtheorem{thm}{Theorem}[section]
\newtheorem{lemm}[thm]{Lemma}
\newtheorem{coro}[thm]{Corollary}
\newtheorem{prop}[thm]{Proposition}

\theoremstyle{definition}

\newtheorem{defi}[thm]{Definition}
\newtheorem{remark}[thm]{Remark}

\begin{document}

\title{Flags and orbits of connected reductive groups over local rings}

\author{Zhe Chen}

\address{Department of Mathematics, Shantou University, Shantou, China}

\email{zhechencz@gmail.com}

\begin{abstract}
We prove that generic higher Deligne--Lusztig representations over truncated formal power series are non-nilpotent, when the parameters are non-trivial on the biggest reduction kernel of the centre; we also establish a relation between the orbits of higher Deligne--Lusztig representations of $\mathrm{SL}_n$ and of $\mathrm{GL}_n$. Then we introduce a combinatorial analogue of Deligne--Lusztig construction for general and special linear groups over local rings; this construction generalises the higher Deligne--Lusztig representations and affords all the nilpotent orbit representations, and for $\mathrm{GL}_n$ it also affords all the regular orbit representations as well as the invariant characters of the Lie algebra.
\end{abstract}

\maketitle

\tableofcontents

\section{Introduction}\label{section:Intro}

Let $\mathcal{O}:=\mathbb{F}_q[[\pi]]$ be a complete discrete valuation ring in characteristic $p$, and let $\mathbb{G}$ be a connected reductive group over $\mathcal{O}$. Here we are interested in the representation theory of $\mathbb{G}(\mathcal{O})$. As a profinite group, each of its smooth representations factors through $\mathbb{G}(\mathcal{O}_r)$ for some $r\in\mathbb{Z}_{>0}$, where $\mathcal{O}_r:=\mathcal{O}/\pi^{r}$; this draws our attention to the representations of $\mathbb{G}(\mathcal{O}_r)$ for $r$ runs over all positive integers.

\vspace{2mm} In the level $1$ case, that is, the case $r=1$, these groups are finite groups of Lie type, and the studies of their representations are to a large extent influenced by the marvellous work of Deligne and Lusztig \cite{DL1976} since 1976. For $r\geq1$, there is a natural generalisation of this work, as pointed out by Lusztig \cite{Lusztig1979SomeRemarks}, by viewing these reductive groups over finite local rings as certain algebraic groups over the finite residue field; some fundamental results on this aspect were first established by Lusztig himself 25 years later in \cite{Lusztig2004RepsFinRings}. However, due to the nature of these algebraic groups (they are usually neither reductive, nor solvable), the corresponding proofs are much more difficult than the $r=1$ case, and the properties of this generalisation are not yet well-understood.

\vspace{2mm} The works of Deligne--Lusztig and Lusztig made an extensive use of \'etale topology of varieties over finite fields, and are referred to as the geometric approach to the representations of $\mathbb{G}(\mathcal{O}_r)$. Meanwhile, there are also algebraic methods, like Clifford theory and orbits, which link the representations of these groups with the orbits in the corresponding finite Lie algebras. Indeed, it is believed that to uniformly describe all irreducible representations of these groups is a very hard, probably hopeless, task, and one shall start by focusing on smaller families of representations; at least for $\mathbb{G}=\mathrm{GL}_n$, the algebraic methods suggest the following three families to be first considered:
\begin{itemize}
\item nilpotent (orbit) representations;
\item regular (orbit) representations;
\item semisimple (orbit) representations.
\end{itemize}
These families are defined in terms of (co-)adjoint orbits via Clifford theory, and it would be very interesting to understand how these algebraically defined families interact with the geometrically constructed representations of Deligne--Lusztig type. This paper gives an attempt on this aspect. In Section~\ref{sec: preliminaries} we review some of the above notions.

\vspace{2mm} Let $\theta$ denote a linear character of $\mathbb{T}(\mathcal{O}_r)$, where $\mathbb{T}$ is a maximal torus of $\mathbb{G}$. When $\theta$ is generic enough, Lusztig associated to it an irreducible representation $R_{\mathbb{T}}^{\theta}$ of $\mathbb{G}(\mathcal{O}_r)$, by viewing $\mathbb{G}$ as an object over $\mathbb{F}_q$; it is called a higher Deligne--Lusztig representation in \cite{ChenStasinski_2016_algebraisation}. Previously, based on Lusztig's explicit computations in \cite{Lusztig2004RepsFinRings}, Stasinski showed in \cite{Sta2011ExtendedDL} that certain nilpotent representations of $\mathrm{SL}_2(\mathbb{F}_q[[\pi]]/\pi^2)$ were unrealisable in higher Deligne--Lusztig theory. In Section~\ref{sec: non-nilpotentcy}, we prove that (when $p$ is not too small), if $\theta$ is non-trivial on the biggest reduction kernel of the centre, then the representation $R_{\mathbb{T}}^{\theta}$ is not nilpotent; see Theorem~\ref{thm: non-nilpotency}. Along the way we also study an inner product via a general principle of Lusztig (see Lemma~\ref{lemm: inner prod}), and give a generalisation of Hill's characterisation of nilpotent representations (see Proposition~\ref{prop: nilpotent reps}). Then in Section~\ref{sec:GL_n and SL_n} we study the higher Deligne--Lusztig representations of $\mathrm{SL}_n$ and $\mathrm{GL}_n$ via the idea of regular embedding, and establish a relation between their orbits (see Proposition~\ref{prop: regularity for SL_n}).

\vspace{2mm} Within the above results, it is desired to find a suitable way to extent the family of higher Deligne--Lusztig representations to cover all the nilpotent representations. The first attempt has been made for $\mathrm{GL}_n$ and $\mathrm{SL}_n$ by Stasinski in \cite{Sta2011ExtendedDL}, in which he introduced an extension using tamely ramified tori, and proved that, for $n=r=2$, all the missing nilpotent representations can be afforded in his construction. In Section~\ref{sec: DL for admissible triple} we give another extension for general and special linear groups; more precisely, we construct an analogue of Deligne--Lusztig representations based on applying Lusztig's idea ``view an object over $\mathcal{O}_r$ as an object over the residue field'' to flags. Our method is different and more combinatorial, reflecting some group theoretical feature appeared in the $r=1$ case. By explicit arguments we show that our construction produces all the higher Deligne--Lusztig representations and affords all the nilpotent representations for $\mathrm{GL}_n$ and $\mathrm{SL}_n$, for any $n$ and $r$; see Theorem~\ref{thm: admissible DL rep}. 

\vspace{2mm} In Section~\ref{sec: reg orbit} we turn back to regular orbits. Using a work of Hill on an analogue of Gelfand--Graev modules, we deduce that our construction mentioned above also affords all the regular orbit representations of $\mathrm{GL}_n(\mathcal{O}_r)$ (see Corollary~\ref{coro: regular orbit}), as well as the so-called invariant characters of the finite Lie algebra.

\vspace{2mm} \noindent {\bf Acknowledgement.} The author is grateful to Alexander Stasinski for several helpful suggestions, and thanks Yongqi Feng and Ying Zong for useful discussions. During the preparation of this work the author is partially supported by the STU funding NTF17021.

\section{Preliminaries}\label{sec: preliminaries}

In this section we recall the notions in higher Deligne--Lusztig representations following \cite{Lusztig2004RepsFinRings}, and recall the notions concerning orbits following \cite[2.4]{Hill_1995_Regular}. 

\vspace{2mm} Let $\mathcal{O}=\mathbb{F}_q[[\pi]]$, where $\pi$ is a uniformiser and $\mathrm{char}~\mathbb{F}_q=p$. Put $\mathcal{O}_r:=\mathcal{O}/\pi^r$, where $r$ is a fixed arbitrary positive integer. Let $\mathbb{G}$ be a connected reductive group over $\mathrm{Spec}~\mathbb{F}_q[[\pi]]/\pi^r$ and let $\mathbf{G}$ be the base change of $\mathbb{G}$ to $\mathrm{Spec}~\overline{\mathbb{F}}_q[[\pi]]/\pi^r$. Every closed subgroup scheme $\mathbf{H}$ of $\mathbf{G}$, like $\mathbf{G}$ itself, Borel subgroups, and maximal tori, admits a natural algebraic group structure over $\mathrm{Spec}~\overline{\mathbb{F}}_q$, denoted by the corresponding Roman font $H$ (sometimes we also write $H_r$ for $H$ when different levels are involved); we have
$$\mathbf{H}(\overline{\mathbb{F}}_q[[\pi]]/\pi^r)=H(\overline{\mathbb{F}}_q)$$
as groups. There is a natural (geometric) Frobenius endomorphism $F$ on $G$ such that 
$$G^F:=G(\overline{\mathbb{F}}_q)^F\cong\mathbb{G}(\mathcal{O}_r)$$
as finite groups. A closed subgroup scheme $\mathbf{H}\subseteq\mathbf{G}$ is called $F$-stable or $F$-rational, if $FH\subseteq H$. We define the Lang isogeny (on $G$) with respect to $F$ to be $L\colon g\mapsto g^{-1}F(g)$.

\vspace{2mm} Fix a prime $\ell\neq p$. For any variety $X$ over $\overline{\mathbb{F}}_p$ (in this paper, by a variety we mean a reduced quasi-projective scheme) we have the compactly supported $\ell$-adic cohomology groups $H^i_c(X,\overline{\mathbb{Q}}_{\ell})$, which are finite dimensional vector spaces over $\overline{\mathbb{Q}}_{\ell}$.

\vspace{2mm} Let $\mathbf{B}=\mathbf{U}\rtimes\mathbf{T}$ be a Borel subgroup of $\mathbf{G}$, where $\mathbf{U}$ is the unipotent radical and $\mathbf{T}$ is an $F$-stable maximal torus; let $B,U,T$ be the corresponding closed subgroup of $G$. Note that $G^F$ and $T^F$ act on the variety $L^{-1}(FU)$ by the left translation and the right translation, respectively; this induces a $G^F\times T^F$-bimodule structure on $H^i_c(L^{-1}(FU),\overline{\mathbb{Q}}_{\ell})$.

\begin{defi}\label{defi: DL rep}
The virtual representation 
$$R_{T,U}^{\theta}:=\sum_{i}(-1)^iH_c^i(L^{-1}(FU),\overline{\mathbb{Q}}_{\ell})_{\theta},$$
where $\theta\in\widehat{T^F}$ is an irreducible $\overline{\mathbb{Q}}_{\ell}$-character of $T^F$, is called a higher Deligne--Lusztig representation of $G^F$; here the subscript $(-)_{\theta}$ means the $\theta$-isotypical part.
\end{defi}

For $1\leq i\leq r$, the reduction map $\mathcal{O}_r\rightarrow \mathcal{O}_{i}$ modulo $\pi^i$ induces an algebraic group surjection 
$$\rho_{r,i}\colon G_r\longrightarrow G_i;$$
we denote the kernel by $G^i:=G_r^i$. Note that all these $G^i$ are $F$-stable. Similar notation applies to the closed subgroups of $G$. Now fix a character $\theta\in \widehat{T^F}$. If the stabiliser of $\theta$ in $N(T)^F$ is $T^F$, then we say $\theta$ is in \emph{general position}. If for any root $\alpha$ of $T$ we have $\theta|_{N_{F}^{F^a}(((T^{\alpha})^{r-1})^{F^a})}\neq 1$, where $N_{F}^{F^a}(t):=t\cdot F(t) \cdots t^{F^{a-1}}$ denotes the norm map and $a\in\mathbb{Z}_{>0}$ is any integer such that $T^{\alpha}$ is $F^a$-stable, then we say $\theta$ is \emph{regular}. We have:

\begin{thm}[Lusztig]\label{thm: irr of generic DL}
If $\theta$ is regular and in general position, then $R_{T,U}^{\theta}$, up to a sign, is an irreducible representation of $\mathbb{G}(\mathcal{O}_r)$.
\end{thm}
\begin{proof}
See \cite[2.4]{Lusztig2004RepsFinRings}.
\end{proof}

In the remaining of this section let $r\geq 2$; note that $G^{r-1}$ can be viewed as the additive group of the Lie algebra $\mathfrak{g}$ of $G_1$, and the conjugation action of $G$ on $G^{r-1}$ becomes the adjoint action of $G_1$ on $\mathfrak{g}$ (as the conjugation action of $G$ on $G^{r-1}$ factors through $G_1$). For any irreducible $\overline{\mathbb{Q}}_{\ell}$-representation $\sigma$ of $G^F$, consider the restriction $\sigma|_{(G^{r-1})^F}$; by Clifford theory (see e.g.\ \cite[6.2]{Isaacs_CharThy_Book}) we have
$$\sigma|_{(G^{r-1})^F}\cong e\cdot\left(\sum_{\chi\in\Omega(\sigma)} \chi \right),$$
where $e\in \mathbb{Z}_{>0}$ and $\Omega(\sigma)$ is a $G_1^F$-orbit of irreducible representations of $(G^{r-1})^F$. In particular, we have a map 
\begin{equation*}
\Omega\colon \mathrm{Irr}(\mathbb{G}(\mathcal{O}_r))\longrightarrow G_1^F\backslash\mathrm{Irr}((G^{r-1})^F)
\end{equation*}
from the set of irreducible representations of $G^F$ to the set of $\mathbb{G}(\mathbb{F}_q)$-orbits of irreducible representations of $\mathfrak{g}^F$.

\vspace{2mm} Consider the following condition, under which we will often work:
\begin{itemize}\hypertarget{condition (*)}{}
\item[{\bf (*)}] There exists a non-degenerate $G_1$-invariant bilinear form  $\mu(-,-)\colon \mathfrak{g}\times\mathfrak{g}\rightarrow \overline{\mathbb{F}}_q$ defined over $\mathbb{F}_q$, as well as a $G_1$-equivariant bijective morphism $\epsilon(-)$ from the nilpotent variety of $\mathfrak{g}$ to the unipotent variety of $G_1$ defined over $\mathbb{F}_q$.
\end{itemize}
This condition requires $p$ is not too small unless $\mathbb{G}=\mathrm{GL}_n$; see e.g.\ the discussions in \cite[2.5 and 2.7]{Let2005book} and \cite[1.15]{Carter1993FiGrLieTy}. 

\vspace{2mm} Now assume the existence and fix such a bilinear form $\mu$; we also fix a non-trivial character $\phi\colon \mathbb{F}_q\rightarrow \overline{\mathbb{Q}}_{\ell}^{\times}$, then for any $y\in \mathfrak{g}^F\cong (G^{r-1})^F$, we have an induced $\overline{\mathbb{Q}}_{\ell}$-character 
$$\phi_y\colon x\longmapsto \phi(\mu(x,y))$$ 
of $ \mathfrak{g}^F\cong (G^{r-1})^F$. The map $y\mapsto \phi_y$ induces a bijection from the set of adjoint $G_1^F$-orbits of $\mathfrak{g}^F$ to the set of $G_1^F$-orbits of irreducible representations of $\mathfrak{g}^F$; combine the above $\Omega$ with the inverse of this bijection we obtain a map
$$\Omega'\colon \mathrm{Irr}(\mathbb{G}(\mathcal{O}_r))\longrightarrow \mathbb{G}(\mathbb{F}_q)\backslash \mathfrak{g}^F$$
from the set of irreducible representations of $G^F$ to the set of adjoint orbits of $\mathfrak{g}^F$.

\begin{defi}
An irreducible representation $\sigma$ of $\mathbb{G}(\mathcal{O}_r)$ is called nilpotent (resp.\ regular, semisimple), if $\Omega'(\sigma)$ is nilpotent (resp.\ regular, semisimple).
\end{defi}

To avoid possible confusions we sometimes also use the term \emph{nilpotent (resp.\ regular, semisimple) orbit representation}.

\section{Inner products and (non-)nilpotency}\label{sec: non-nilpotentcy}

Throughout this section we assume \hyperlink{condition (*)}{{\bf (*)}} holds, and fix the choice of such a $\mu$ and $\epsilon$.

\vspace{2mm} Based on Lusztig's computations in \cite{Lusztig2004RepsFinRings}, Stasinski proved in \cite{Sta2011ExtendedDL} that certain nilpotent representations of $\mathrm{SL}_2(\mathbb{F}_q[[\pi]]/\pi^2)$ are missed in higher Deligne--Lusztig theory. In this section we continue this direction, and prove a result suggesting that, for any group with a non-trivial connected centre, the intersection of the set of higher Deligne--Lusztig representations and the set of nilpotent representations is always small.

\vspace{2mm} Let $\mathbf{B}'=\mathbf{U}'\rtimes \mathbf{T}'$ be another Borel subgroup of $\mathbf{G}$, with $\mathbf{T}'$ an $F$-stable maximal torus and $\mathbf{U}'$ the unipotent radical; let $B'=U'\rtimes T'$ be the corresponding subgroups of $G$. Consider the $G^F$-representation $\sum_{i}(-1)^iH_c^{i}(L^{-1}(FB'^{r-1})$; it appears in the scope of a generalised Deligne--Lusztig construction (see the argument of Theorem~\ref{thm: admissible DL rep}~(ii)). Here we first prove the following general inner product result.

\begin{lemm}\label{lemm: inner prod} 
Suppose $\mathcal{Z}:=(Z(G)^{1})^{\circ}\neq 1$, then we have
\begin{equation*}
\left \langle \sum_{i}(-1)^iH_c^{i}(L^{-1}(FB'^{r-1}),\overline{\mathbb{Q}}_{\ell}), R_{T,U}^{\theta}\right\rangle_{G^F}=0
\end{equation*}
if $\theta$ is non-trivial on $\mathcal{Z}^F$.
\end{lemm}
\begin{proof}
By \cite[3.1]{Chen_2017_InnerProduct} the inner product equals to the Euler characteristic
\begin{equation*}
\sum_i(-1)^i\dim H^i_c(\Sigma,\overline{\mathbb{Q}}_{\ell})_{\theta},
\end{equation*}
where 
$$\Sigma:=\{  (x',x,y)\in FB'^{r-1}\times FU\times G\mid x'F(y)=yx \},$$
on which $t\in T^F$ acts by sending $(x',x,y)$ to $(x',t^{-1}xt,yt)$. 

\vspace{2mm} Write $y=\hat{y}y_1$ and $x=\hat{x}x_1$, where $\hat{y}\in G^1$, $\hat{x}\in FU^1$, $y_1=\rho_{r,1}(y)\in G_1$, and $x_1=\rho_{r,1}(x)\in FU_1$. Applying $\rho_{r,1}$ to the defining equation $x'F(y)=yx$ of $\Sigma$ we see that $x_1=y_1^{-1}F(y_1)$, so 
$$x'F(y)=x'F(\hat{y})F(y_1)=x'F(\hat{y})y_1x_1$$
and
$$yx=\hat{y}y_1\hat{x}x_1,$$
thus $\Sigma$ can be re-written as
$$\Sigma'=\{  (x',\hat{x},x_1, \hat{y},y_1)\in FB'^{r-1}\times FU^1\times FU_1 \times G^1\times  G_1)  \mid x'F(\hat{y})y_1=\hat{y}y_1\hat{x}; L(y_1)=x_1 \}.$$
The identification $\Sigma=\Sigma'$ is $T^F$-equivariant under the action of $T^F$ on $\Sigma'$ given by: 
\begin{equation*}
\begin{tikzcd}
(x',\hat{x},x_1, \hat{y},y_1)\in\Sigma' \arrow[d,mapsto,swap,"t\in T^F"]  \\
(x',t^{-1}\hat{x}x_1t \rho_{r,1}(t^{-1}){x}_1^{-1}\rho_{r,1}(t),  \rho_{r,1}(t)^{-1}x_1\rho_{r,1}(t),\hat{y}y_1t\rho_{r,1}(y_1t)^{-1},\rho_{r,1}(y_1t) ) \in\Sigma'.
\end{tikzcd}
\end{equation*}
In particular, $(T^1)^F$ acts on $\Sigma'$ by
$$t\in (T^1)^F\colon (x',\hat{x},x_1, \hat{y},y_1)\longmapsto (x',t^{-1}\hat{x}x_1t x_1^{-1}, x_1, \hat{y}y_1ty_1^{-1},y_1 ).$$
Now consider the restriction of the above action of $(T^1)^F$ to the subgroup $\mathcal{Z}^F$ (note that $F\mathcal{Z}=\mathcal{Z}$ by definition); this restriction extends to the whole connected algebraic group $\mathcal{Z}$ via:
$$(x',\hat{x},x_1, \hat{y},y_1)\longmapsto (x'f(t),t^{-1}\hat{x}x_1t x_1^{-1}, x_1, \hat{y}y_1ty_1^{-1},y_1 )$$
for any $t\in \mathcal{Z}$, where $f(t):=tF(t)^{-1}$. Therefore, by the homotopy property of \'etale cohomology \cite[Page~136]{DL1976}, the action of $\mathcal{Z}^F$ on $H_c^i(\Sigma,\overline{\mathbb{Q}}_{\ell})$ is trivial, thus $H_c^i(\Sigma,\overline{\mathbb{Q}}_{\ell})_{\theta}=0$ by our assumption. This proves the lemma.
\end{proof}

We need the following characterisation of nilpotent representations.

\begin{prop}\label{prop: nilpotent reps}
If $\mathbf{B}'$ is $F$-stable, then the nilpotent representations of $\mathbb{G}(\mathcal{O}_r)$ are the irreducible components of $\mathrm{Ind}_{(B'^{r-1})^F}^{G^F}\mathrm{1}$.
\end{prop}
\begin{proof}
This was originally proved in the case $\mathbb{G}=\mathrm{GL}_n$ (though, with $\mathcal{O}=\mathbb{Z}_p$) by Hill in \cite[2.12]{Hill_1993_Jordan}; we would generalise his argument to our situation.

\vspace{2mm} Let $\mathfrak{b}'$, $\mathfrak{u}'$, and $\mathfrak{t}'$ be the Lie algebra of $B'_1$, $U'_1$, and $T'_1$, respectively; we also identify them with $B'^{r-1}$, $U'^{r-1}$, and $T'^{r-1}$. We first prove that
\begin{equation}\label{temp:bilinear form}
\mu(\mathfrak{u}',\mathfrak{b}')=0.
\end{equation}
By linearity it suffices to show that $\mu(\mathfrak{u}'_{\alpha},\mathfrak{u}'_{\beta})=0$ and $\mu(\mathfrak{u}'_{\alpha},\mathfrak{t}')=0$ for any positive roots $\alpha$ and $\beta$ of $T'_1$ with respect to $B'_1$, where $\mathfrak{u}'_{\alpha}$ denote the corresponding Lie algebra of the root subgroup $(U'_1)_{\alpha}$. Let $\mathrm{Ad}(g)$ denotes the adjoint action of $g\in G_1$ on $\mathfrak{g}$. Then for any $t'\in T'_1$, $\tau'\in \mathfrak{t}'$, $u\in \mathfrak{u}'_{\alpha}$, and $v\in \mathfrak{u}'_{\beta}$ we have:
$$\mu(u,\tau')=\mu(\mathrm{Ad}(t')u,\mathrm{Ad}(t')\tau')=\alpha(t')\mu(u,\tau')$$
and
$$\mu(u,v)=\mu(\mathrm{Ad}(t')u,\mathrm{Ad}(t')v)=\alpha(t')\beta(t')\mu(u,v),$$
which imply that 
$$\mu(\mathfrak{u}'_{\alpha},\mathfrak{u}'_{\beta})=0=\mu(\mathfrak{u}'_{\alpha},\mathfrak{t}'),$$
so \eqref{temp:bilinear form} holds.

\vspace{2mm} By \eqref{temp:bilinear form}, for any $y\in \mathfrak{u}'^F$ we have $\phi_y|_{(B'^{r-1})^F}=1$ (see the notation in Section~\ref{sec: preliminaries}), thus Frobenius reciprocity and dimension counting imply that
$$\mathrm{Ind}_{(B'^{r-1})^F}^{(G^{r-1})^F}1=\bigoplus_{y\in \mathfrak{u}'^F}\phi_y,$$
hence
\begin{equation}\label{temp: final formula in nilp lemma}
\mathrm{Ind}_{(B'^{r-1})^F}^{G^F}1=\bigoplus_{y\in \mathfrak{u}'^F}\mathrm{Ind}_{(G^{r-1})^F}^{G^F}\phi_y.
\end{equation}
Now let $\nu\in\mathfrak{g}^F$ be a nilpotent element, then by \cite[3.20]{DM1991} and \cite[14.26]{Borel_1991_LinearAlgGp} (see also Grothendieck's covering theorem \cite[XIV~4.11]{SGA3}) there is a $\upsilon\in (U'')^F$ such that $\nu=\epsilon(\upsilon)$, where $U''$ is the unipotent radical of an $F$-stable Borel subgroup $B''$ of $G_1$. Meanwhile, as any two Borel subgroups are conjugate, there is a $g\in G_1$ with $gB''g^{-1}=B'_1$; we shall show that such a $g$ can be chosen to be an element in $G_1^F$. Indeed, by the $F$-rationality of $B'$ and $B''$ we have 
$$F(g)B''F(g^{-1})=gB''g^{-1},$$
thus $g^{-1}F(g)\in B''$ by the self-normality of Borel subgroups, so, according to the Lang--Steinberg theorem there is a $b\in B''$ such that $g^{-1}F(g)=b^{-1}F(b)$. In particular, $F(gb^{-1})=gb^{-1}$ and $(gb^{-1})B''(gb^{-1})^{-1}=B'_1$, thus every nilpotent $G_1^F$-orbit of $\mathfrak{g}^F$ intersects $\mathfrak{u}'^F$, so the proposition follows from \eqref{temp: final formula in nilp lemma}. 
\end{proof}

Now we can prove:

\begin{thm}\label{thm: non-nilpotency}
Let $\mathcal{Z}$ be as in Lemma~\ref{lemm: inner prod}, and let $\theta$ be regular and in general position. If $\theta$ is non-trivial on $\mathcal{Z}^F$, then the higher Deligne--Lusztig representation $R_{T,U}^{\theta}$, up to a sign, is not nilpotent.
\end{thm}
\begin{proof}
In this proof we assume that $\mathbf{B}'$ is $F$-stable. It is known that one can always find such a $\mathbf{B}'$; for example, this follows from the Lang--Steinberg theorem (see e.g.\ the argument in \cite[1.17]{Carter1993FiGrLieTy}) via the formal smoothness of strict transporters (see \cite[4.15]{Sta2012ReductiveGr} and \cite[XXII 5.3.9]{SGA3}). 

\vspace{2mm} Consider the morphism $L^{-1}(FB'^{r-1})=L^{-1}(B'^{r-1})\rightarrow G/B'^{r-1}$ given by
$$x\longmapsto xB'^{r-1}.$$
Note that its image is $G^F/(B'^{r-1})^F$, and the fibres are isomorphic to the affine space $B'^{r-1}$, thus, as $G^F$-modules we have
$$\sum_{i}(-1)^iH_c^{i}(L^{-1}(FB'^{r-1})\cong \overline{\mathbb{Q}}_{\ell}[G^F/(B'^{r-1})^F]=\mathrm{Ind}_{(B'^{r-1})^F}^{G^F}\mathrm{1}.$$
Now the assertion follows from Lemma~\ref{lemm: inner prod}, Proposition~\ref{prop: nilpotent reps}, and Theorem~\ref{thm: irr of generic DL}.
\end{proof}

Note that the above result also holds for the Deligne--Lusztig-type representations $R_{T,U,b}^{\theta}$ considered in \cite{Chen_2017_InnerProduct}, via an almost same argument.

\section{Orbits of general and special linear groups}\label{sec:GL_n and SL_n}

Throughout this section let $\mathbb{G}=\mathrm{SL}_n$ with $p\nmid n$ (note that in this case $Z(G)^1$ is trivial, so Theorem~\ref{thm: non-nilpotency} does not work), and take  $\mu$ to be the trace form $\mathrm{Tr}(-\cdot-)$ (which is non-degenerate on $G^{r-1}$ as $p\nmid n$).

\vspace{2mm} Let $\widetilde{\mathbb{G}}$ be $\mathrm{GL}_n$ over $\mathcal{O}_r$, and let $\widetilde{\mathbf{G}}$ be its base change to $\mathrm{Spec}~\overline{\mathbb{F}}_q[[\pi]]/\pi^r$, then we have a natural closed immersion $i\colon \mathbf{G}\rightarrow\widetilde{\mathbf{G}}$. Let $\widetilde{\mathbf{B}}=\widetilde{\mathbf{U}}\rtimes \widetilde{\mathbf{T}}$ be the Borel subgroup of $\widetilde{\mathbf{G}}$ such that $i^{-1}(\widetilde{\mathbf{B}})=\mathbf{B}$, $i^{-1}(\widetilde{\mathbf{T}})=\mathbf{T}$, and $i^{-1}(\widetilde{\mathbf{U}})=\mathbf{U}$, where $\widetilde{\mathbf{U}}$ is the unipotent radical and $\widetilde{\mathbf{T}}$ a maximal torus. Let $\widetilde{G}$, $\widetilde{B}$, $\widetilde{U}$, and $\widetilde{T}$ be the corresponding algebraic groups (note that we have $\widetilde{U}\cong U$ and $F\widetilde{T}=\widetilde{T}$); we again use $i$ to denote the closed immersion $G\rightarrow \widetilde{G}$. The natural geometric Frobenius and the associated Lang isogeny, on $\widetilde{G}$, will still be denoted by $F$ and $L$, respectively. 

\vspace{2mm} Motivated by the regular embedding trick (see \cite{Lusztig_1988_Rep_red_gp_disconnectedcentre} and \cite{Geck_Malle_2016_1stchapter}) we give the following relation between $R_{T,U}^{\theta}$ and $R_{\widetilde{T}, \widetilde{U}}^{\widetilde{\theta}}$.

\begin{prop}\label{prop: regularity for SL_n}
Let $\theta$ be regular and in general position, then, up to a sign, the orbits $\Omega'(R_{T,U}^{\theta})$ and $\Omega'(R_{\widetilde{T}, \widetilde{U}}^{\widetilde{\theta}})$, viewed as subsets of $M_n({\mathbb{F}}_q)$, are different up to a shift by a scalar matrix. In particular, $R_{T,U}^{\theta}$, up to a sign, is a regular orbit representation if and only if $R_{\widetilde{T}, \widetilde{U}}^{\widetilde{\theta}}$ is so for some $\widetilde{\theta}$ restricting to $\theta$.
\end{prop}
\begin{proof}
We first prove $R_{T,U}^{\theta}=\mathrm{Res}^{\widetilde{G}^F}_{G^F}R_{\widetilde{T}, \widetilde{U}}^{\widetilde{\theta}}$ following the argument of \cite[13.22]{DM1991}. Let $g\in G^F=\mathrm{SL}_n(\mathbb{F}_q[[\pi]]/\pi^r)$, then by Brou\'e's bimodule induction formula (see \cite[Chapter~4]{DM1991}) we have
\begin{equation*}
\mathrm{Tr}(R_{T,U}^{\theta},g)=\frac{1}{|T^F|}\cdot \sum_{t\in T^F}\theta(t^{-1})\cdot \mathrm{Tr}\left((g,t)  \mid \sum_{i}(-1)^iH_c^i(L^{-1}(FU),\overline{\mathbb{Q}}_{\ell})  \right)
\end{equation*}
and
\begin{equation*}
\mathrm{Tr}(\mathrm{Res}^{\widetilde{G}^F}_{G^F}R_{\widetilde{T}, \widetilde{U}}^{\widetilde{\theta}},g)=\frac{1}{|\widetilde{T}^F|}\cdot \sum_{t'\in \widetilde{T}^F}\widetilde{\theta}(t'^{-1})\cdot \mathrm{Tr}\left((g,t')  \mid \sum_{i}(-1)^iH_c^i(L^{-1}(F\widetilde{U}),\overline{\mathbb{Q}}_{\ell})  \right).
\end{equation*}
By the formal power series interpretation of Lefschetz numbers (see \cite[1.2]{Lusztig_whiteBk}), to prove that these two values are the same, it suffices to show that:
\begin{equation}\label{temp: SLn 1}
\frac{1}{|T^F|}\sum_{t\in T^F} \theta(t^{-1}) \#\{ x\in G\mid L(x)\in FU,\ gF^d(x)t=x \}
\end{equation}
equals to
\begin{equation}\label{temp: SLn 2}
\frac{1}{|\widetilde{T}^F|}\sum_{t'\in \widetilde{T}^F} \widetilde{\theta}(t'^{-1}) \#\{ x\in \widetilde{G}\mid L(x)\in F\widetilde{U},\ i(g)F^d(x)t'=x \},
\end{equation}
for every $d\in\mathbb{Z}_{>0}$ such that $U$ and $\widetilde{U}$ are defined over $\mathbb{F}_{q^d}$. Given $x\in \widetilde{G}$ with $L(x)\in F\widetilde{U}$, consider the pairs $(y,\lambda)\in i(G)\times \widetilde{T}^F$ such that $y=x\lambda$. There are $|T^F|$ such pairs, so we can rewrite \eqref{temp: SLn 2} as
\begin{equation*}
\frac{1}{|\widetilde{T}^F|}\frac{1}{|T^F|}\sum_{t'\in \widetilde{T}^F,\ \lambda\in \widetilde{T}^F} \widetilde{\theta}(t'^{-1}) \#\{ y\in i(G)\mid L(y)\in F\widetilde{U},\ i(g)F^d(y\lambda^{-1})t'\lambda=y \},
\end{equation*}
which, by the variable change $t''\mapsto \lambda^{-1}t'\lambda$, equals to
\begin{equation*}
\frac{1}{|T^F|}\sum_{t''\in \widetilde{T}^F} \widetilde{\theta}(t''^{-1}) \#\{ y\in i(G) \mid L(y)\in F\widetilde{U},\ i(g)F^d(y)t''=y \}.
\end{equation*}
This last sum equals to \eqref{temp: SLn 1}, because $i(g)F^d(y)t''=y$ implies $t'' \in i(G)\cap\widetilde{T}=T$. Therefore $R_{T,U}^{\theta}=\mathrm{Res}^{\widetilde{G}^F}_{G^F}R_{\widetilde{T}, \widetilde{U}}^{\widetilde{\theta}}$, as desired.

\vspace{2mm} Now write $\mathrm{Res}^{\widetilde{G}^F}_{(\widetilde{G}^{r-1})^F}R_{\widetilde{T}, \widetilde{U}}^{\widetilde{\theta}}$ as the sum of $\widetilde{G}^F$-orbit of $\phi(\mathrm{Tr}(-\cdot z_1))$ for some non-nilpotent $z_1\in M_n({\mathbb{F}}_q)$, and write $\mathrm{Res}^{{G}^F}_{({G}^{r-1})^F}R_{{T}, {U}}^{{\theta}}$ as the sum of ${G}^F$-orbit of $\phi(\mathrm{Tr}(-\cdot z_2))$ for some traceless $z_2\in M_n({\mathbb{F}}_q)$, then $R_{T,U}^{\theta}=\mathrm{Res}^{\widetilde{G}^F}_{G^F}R_{\widetilde{T}, \widetilde{U}}^{\widetilde{\theta}}$ implies 
$$\mathrm{Tr}(z(hz_1h^{-1}-z_2))=0$$ 
for all traceless $z\in M_n({\mathbb{F}}_q)$ for some $h\in \mathrm{GL}_n({\mathbb{F}}_q)$. So by direct computations we see $hz_1h^{-1}=z_2+c$ for some $c=\mathrm{diag}(a,...,a)\in M_n({\overline{\mathbb{F}}}_q)$. In particular, the centraliser of $z_1$ in $G_1$ and the centraliser of $z_2$ in $G_1$ are isomorphic, which implies the theorem.
\end{proof}

Combine the above theorem with the results in \cite{ChenStasinski_2016_algebraisation} we immediately see that, when $r$ is even and $\mathbf{T}$ is Coxeter, the higher Deligne--Lusztig representation $R_{T,U}^{\theta}$ of $\mathrm{SL}_n(\mathbb{F}_q[[\pi]]/\pi^r)$ is regular and semisimple if $\theta$ is regular and in general position.

\section{Flags and Deligne--Lusztig constructions}\label{sec: DL for admissible triple}

In the case $r=1$, a Deligne--Lusztig variety is by definition the Lang isogeny twist of a Bruhat cell, that is, the variety $L^{-1}(B_0wB_0)/B_0$, where $B_0$ is an $F$-stable Borel subgroup of $G$ and $w$ is an element in the Weyl group of an $F$-stable maximal torus $T_0\subseteq B_0$; while this definition admits fruitful geometric/combinatorial flavour, in the study of its cohomology spaces, one usually passes to certain affine bundle over a $T_0^{wF}$-torsor of the variety. More precisely, in the study of cohomology as a representation space one usually focuses on the variety $L^{-1}(FU)$, which can be described simply as:
\begin{itemize}
\item[(P)] The preimage of the unipotent radical of a (not necessarily $F$-stable) Borel subgroup along a Lang map.
\end{itemize}

\vspace{2mm} In the general $r\geq 1$ case, we followed this treatment and took the Deligne--Lusztig representations to be $\sum_i(-1)^iH_c^i(L^{-1}(FU),\overline{\mathbb{Q}}_{\ell})_{\theta}$, as in Definition~\ref{defi: DL rep}. However, unless $r=1$, property~(P) is no more true at the level of algebraic groups, namely, the unipotent group $U=U_r$ is no more the unipotent radical of $B=B_r$ (as $T=T_r\cong B/U$ has a non-trivial unipotent radical $T^1$).

\vspace{2mm} In the remaining part of this section we assume $\mathbb{G}=\mathrm{GL}_n$ or $\mathrm{SL}_n$. 

\vspace{2mm} We want to find a natural analogue of Deligne--Lusztig construction for $\mathbb{G}$, such that it fulfils the following four requirements in a suitable way:
\begin{itemize}
\item The involved family of varieties contains the varieties $L^{-1}(FU)$ for all $U$;
\item all the nilpotent representations can be afforded by the cohomologies of the involved varieties;
\item property~(P) holds in a natural way;
\item if $r=1$, then it coincides with the classical Deligne--Lusztig theory \cite{DL1976}.
\end{itemize}
Note that for general and special linear groups over an algebraically closed field, a Borel subgroup can be viewed as the stabiliser of a complete flag over the field; one of the starting points of our construction is to simulate Lusztig's idea ``view an object over ${\mathcal{O}}_r$ as an object over the residue field ${\mathbb{F}}_q$'' for complete flags.

\vspace{2mm} Consider the $\overline{\mathbb{F}}_q[[\pi]]/\pi^r$-module $(\overline{\mathbb{F}}_q[[\pi]]/\pi^r)^n$, on which $G$ acts in a natural way. Here we view $(\overline{\mathbb{F}}_q[[\pi]]/\pi^r)^n$ as a vector space over $\overline{\mathbb{F}}_q$. By an $\overline{\mathbb{F}}_q$-flag $\mathcal{F}$ of $(\overline{\mathbb{F}}_q[[\pi]]/\pi^r)^n$ we mean a sequence of linear subspaces over $\overline{\mathbb{F}}_q$:
$$\{ 0 \}=:V_0\subseteq V_1\subseteq ... \subseteq V_{nr}:=(\overline{\mathbb{F}}_q[[\pi]]/\pi^r)^n$$
with $\dim_{\overline{\mathbb{F}}_q} V_i=i$. In particular, here we restrict to the complete flags. We call $\mathcal{F}$ \emph{admissible} if its stabiliser in $G$ is non-reductive; when this is the case we denote by $B_{\mathcal{F}}$ its stabiliser and $U_{\mathcal{F}}$ the unipotent radical of $B_{\mathcal{F}}$. 

\vspace{2mm} Now consider the triple $(\mathcal{F},C,\theta)$, where $\mathcal{F}$ is an admissible flag, $C$ an $F$-stable Cartan subgroup of $N_G(U_{\mathcal{F}})^{\circ}$ (recall that a Cartan subgroup is by definition the centraliser of a maximal torus), and $\theta\in\widehat{C^F}$ an irreducible $\overline{\mathbb{Q}}_{\ell}$-character of $C^F$; we call such a triple an \emph{admissible} triple. Note that $L^{-1}(FU_{\mathcal{F}})$ admits the left translation action of $G^F=\mathbb{G}(\mathcal{O}_r)$ and the right translation action of $C^F$.

\begin{defi}
For an admissible triple $(\mathcal{F},C,\theta)$, we put 
$$R_{\mathcal{F},C}^{\theta}:=\sum_i(-1)^{i}H_c^i(L^{-1}(FU_{\mathcal{F}}),\overline{\mathbb{Q}}_{\ell})\otimes_{\overline{\mathbb{Q}}_{\ell}[C^F]}\theta,$$
which is a virtual representation of $\mathbb{G}(\mathcal{O}_r)$.
\end{defi}

In the below, we show by an explicit argument that $R_{\mathcal{F},C}^{\theta}$ yields all the higher Deligne--Lusztig representations and affords the nilpotent representations; in particular, this construction fulfils the above four requirements.

\begin{thm}\label{thm: admissible DL rep}
We have:
\begin{itemize}
\item[(i)] If $\mathbb{G}=\mathrm{SL}_n$ or $\mathrm{GL}_n$, then each higher Deligne--Lusztig representation $R_{T,U}^{\theta}$ is the representation $R_{\mathcal{F},C}^{\theta}$ for an admissible triple $(\mathcal{F},C,\theta)$;
\item[(ii)] Suppose either $\mathbb{G}=\mathrm{SL}_n$ and $p\nmid n$, or $\mathbb{G}=\mathrm{GL}_n$ (so that \hyperlink{condition (*)}{{\bf (*)}} holds), then each nilpotent representation is an irreducible constituent of $R_{\mathcal{F},C}^{\theta}$ for some admissible triple $(\mathcal{F},C,\theta)$.
\end{itemize}
\end{thm}

\begin{proof}
We can assume $r\geq 2$. Let us start with the following observation:
\begin{lemm}\label{lemm: matrix injection}
The map $M_n(\overline{\mathbb{F}}_q[[\pi]]/\pi^r)\rightarrow {M}_{nr}(\overline{\mathbb{F}}_q)$ given by
\begin{equation*}
A_0+A_1\pi+...+A_{r-1}\pi^{r-1}\longmapsto
\begin{bmatrix}
A_0 & 0 &  0 & 0 & ... & 0\\
A_1 & A_0 & 0 & 0 & ... & 0\\
A_2 & A_1 & A_0 & 0 & ... & 0\\
... & ... & ...  & ... & ... & ... \\
A_{r-2} & ... & ...  & A_1 & A_0 & 0 \\
A_{r-1} & A_{r-2} & ... & ... & A_1 & A_0
\end{bmatrix},
\end{equation*}
where $A_i\in M_n(\overline{\mathbb{F}}_q)$, is an injective ring morphism. Moreover, by taking inverse limits, this map extends to a map $M_n(\overline{\mathbb{F}}_q[[\pi]])\rightarrow M_{\infty}(\overline{\mathbb{F}}_q)$ given by
\begin{equation*}
A_0+A_1\pi+A_2\pi^2...\longmapsto
\begin{bmatrix}
A_0 & 0 &  0 & 0 & ... & 0\\
A_1 & A_0 & 0 & 0 & ... & 0\\
A_2 & A_1 & A_0 & 0 & ... & 0\\
... & ... & ...  & ... & ... & ... \\
... & ... & ...  & ... & ... & ... \\
\end{bmatrix},
\end{equation*}
preserving the addition and multiplication.
\end{lemm}
\begin{proof}
This follows from direct computations.
\end{proof}
Now we view $G=G_r$ as a closed subgroup of $\mathrm{GL}_{nr}$ via the above injection, where $\mathrm{GL}_{nr}$ is understood as the automorphism group of the $\overline{\mathbb{F}}_q$-vector space with the standard basis 
$$\{x^{(0)}_1,x^{(0)}_2,...,x^{(0)}_n,x^{(1)}_1,x^{(1)}_2,...,x^{(r-1)}_n\}.$$
We want to prove (i) and (ii) by constructing explicit flags via this basis.

\vspace{2mm} (i): We shall first construct a flag $\mathcal{F}$ with $U_{\mathcal{F}}=U$. We only need to do this by assuming $\mathbf{B}$ is the standard upper Borel subgroup of $\mathbf{G}$: Indeed, if $\mathbf{B}'$ is another Borel subgroup and $B'$ the corresponding algebraic group, then $B$ and $B'$ are conjugate in $G$; for example, this follows from the formal smoothness of strict transporters (for the details, see \cite[4.15]{Sta2012ReductiveGr} and \cite[XXII 5.3.9]{SGA3}).

\vspace{2mm} Now we construct $\mathcal{F}$ reversely as
\begin{equation*}
\begin{split}
&\langle x^{(0)}_1,x^{(0)}_2,...,x^{(0)}_n,x^{(1)}_1,x^{(1)}_2,...,x^{(r-1)}_{n-1},x^{(r-1)}_n \rangle\\
&\supseteq \langle x^{(0)}_1,x^{(0)}_2,...,x^{(0)}_n,x^{(1)}_1,x^{(1)}_2,...,...,x^{(r-1)}_{n-1},0 \rangle\\
&\supseteq \langle x^{(0)}_1,x^{(0)}_2,...,x^{(0)}_n,x^{(1)}_1,x^{(1)}_2,...,x^{(r-2)}_{n-1},0,x^{(r-1)}_1,...,x^{(r-1)}_{n-1},0 \rangle\\
&\supseteq   ... \\
&\supseteq \langle x^{(0)}_1,x^{(0)}_2,...,x^{(0)}_{n-1},0,x^{(1)}_1,x^{(1)}_2,...,x^{(r-2)}_{n-1},0,x^{(r-1)}_1,...,x^{(r-1)}_{n-1},0 \rangle\\
&\supseteq  \langle x^{(0)}_1,x^{(0)}_2,...,x^{(0)}_{n-1},0,x^{(1)}_1,x^{(1)}_2,...,x^{(r-2)}_{n-1},0,x^{(r-1)}_1,...,x^{(r-1)}_{n-2},0,0 \rangle \\
&\supseteq  ... \\
&\supseteq  \langle x^{(0)}_1, 0, ... ,0 \rangle \supseteq  \{0\}.
\end{split}
\end{equation*}
By direct computations with the aid of Lemma~\ref{lemm: matrix injection}, we see  $B_{\mathcal{F}}=T_1U$, whose unipotent radical is $U_{\mathcal{F}}=U$, as desired.

\vspace{2mm} Moreover, $T$ is a Cartan subgroup of $G$ by \cite[4.13]{Sta2012ReductiveGr}, hence a Cartan subgroup of the group $N_G(U_{\mathcal{F}})$. In particular, $(\mathcal{F},T,\theta)$ is an admissible triple. This proves (i).

\vspace{2mm} (ii): Again let $\mathbf{B}$ be the standard upper Borel subgroup. From Proposition~\ref{prop: nilpotent reps} and the argument of Theorem~\ref{thm: non-nilpotency} it is sufficient to construct an admissible flag $\mathcal{F}$ with an $F$-stable $U_{\mathcal{F}}\subseteq B^{r-1}$, as then we would have
$$\mathrm{Ind}_{(B^{r-1})^F}^{G^F}1\subseteq\mathrm{Ind}_{U_{\mathcal{F}}^F}^{G^F}1=\sum_{\theta\in\widehat{C^F}}R_{\mathcal{F},C}^{\theta}.$$
We divide the situation into two cases, $r\geq 3$ and $r=2$; in the case $r\geq 3$ we will construct an $\mathcal{F}$ such that the equality $U_{\mathcal{F}}= B^{r-1}$ holds.

\vspace{2mm} Case-I ($r\geq 3$): We construct $\mathcal{F}$ reversely through three steps as follows.

\vspace{2mm} The first step is the following partial flag $\mathcal{F}'$:
\begin{equation*}
\begin{split}
&\langle x^{(0)}_1,x^{(0)}_2,...,x^{(0)}_n,x^{(1)}_1,x^{(1)}_2,...,x^{(r-1)}_n \rangle\\
&\supseteq\langle x^{(0)}_1,x^{(0)}_2,...,x^{(0)}_n,x^{(1)}_1,x^{(1)}_2,...,x^{(r-3)}_{n},0,x^{(r-2)}_2,...,x^{(r-2)}_n,x^{(r-1)}_1,...,x^{(r-1)}_n \rangle\\
&\supseteq\langle x^{(0)}_1,x^{(0)}_2,...,x^{(0)}_n,x^{(1)}_1,x^{(1)}_2,...,x^{(r-3)}_{n},0,0,x^{(r-2)}_3,...,x^{(r-2)}_n,x^{(r-1)}_1,...,x^{(r-1)}_n \rangle\\
&\supseteq  ... \\
&\supseteq \langle x^{(0)}_1,x^{(0)}_2,...,x^{(0)}_n,x^{(1)}_1,x^{(1)}_2,...,x^{(r-3)}_{n},0,...,0,x^{(r-1)}_1,...,x^{(r-1)}_n \rangle.
\end{split}
\end{equation*}
By direct computations with Lemma~\ref{lemm: matrix injection} we see that its stabiliser can be described as: The $A_0$-part is the lower Borel subgroup, and all the other $A_i$'s are zero matrices except for $A_{r-1}$, which runs over the whole $\mathfrak{g}$. 

\vspace{2mm} The second step is the following partial flag $\mathcal{F}''$:
\begin{equation*}
\begin{split}
&\langle x^{(0)}_1,x^{(0)}_2,...,x^{(0)}_n,x^{(1)}_1,x^{(1)}_2,...,x^{(r-3)}_{n},0,...,0,x^{(r-1)}_1,...,x^{(r-1)}_n \rangle\\
&\supseteq  ... \\
&\supseteq\langle x^{(0)}_1,x^{(0)}_2,...,x^{(0)}_n,0,...,0,x^{(r-1)}_1,...,x^{(r-1)}_n \rangle,
\end{split}
\end{equation*}
in which we delete the $x^{(i)}_j$'s with $0<i\leq r-3$ from right to left. The stabiliser of the composite partial flag $\mathcal{F}'\supseteq\mathcal{F}''$ equals the stabiliser of $\mathcal{F}'$.

\vspace{2mm} The third step is the following partial flag $\mathcal{F}'''$:
\begin{equation*}
\begin{split}
&\langle x^{(0)}_1,x^{(0)}_2,...,x^{(0)}_n,0,...,0,x^{(r-1)}_1,...,x^{(r-1)}_{n-1},x^{(r-1)}_n \rangle\\
&\supseteq \langle x^{(0)}_1,x^{(0)}_2,...,x^{(0)}_{n-1},0,0,...,0,x^{(r-1)}_1,...,x^{(r-1)}_{n-1},x^{(r-1)}_n \rangle\\
&\supseteq \langle x^{(0)}_1,x^{(0)}_2,...,x^{(0)}_{n-1},0,0,...,0,x^{(r-1)}_1,...,x^{(r-1)}_{n-1},x^{(r-1)}_{n-1},0 \rangle\\
&\supseteq \langle x^{(0)}_1,x^{(0)}_2,...,x^{(0)}_{n-2},0,0,0,...,0,x^{(r-1)}_1,...,x^{(r-1)}_{n-1},x^{(r-1)}_{n-1},0 \rangle\\
&\supseteq \langle x^{(0)}_1,x^{(0)}_2,...,x^{(0)}_{n-2},0,0,0,...,0,x^{(r-1)}_1,...,x^{(r-1)}_{n-1},x^{(r-1)}_{n-2},0,0 \rangle\\
&\supseteq  ... \\
&\supseteq \langle  0, ... ,0,x^{(r-1)}_1,0,...,0 \rangle\supseteq\{0\}.
\end{split}
\end{equation*}
Now take $\mathcal{F}$ to be the composite $\mathcal{F}'\supseteq\mathcal{F}''\supseteq\mathcal{F}'''$, then by direction computations through Lemma~\ref{lemm: matrix injection} we see that $\mathcal{F}$ is an admissible flag with $B_{\mathcal{F}}=T_1B^{r-1}$ and $U_{\mathcal{F}}=B^{r-1}$, as desired.

\vspace{2mm} Case-II ($r=2$): Let $\mathcal{F}$ be the following flag:
\begin{equation*}
\begin{split}
&\langle x^{(0)}_1,x^{(0)}_2,...,x^{(0)}_n,x^{(1)}_1,x^{(1)}_2,...,x^{(1)}_n \rangle\\
&\supseteq\langle x^{(0)}_1,x^{(0)}_2,...,x^{(0)}_{n-1},0,x^{(1)}_1,...,x^{(1)}_{n-1},x^{(1)}_n \rangle\\
&\supseteq \langle x^{(0)}_1,x^{(0)}_2,...,x^{(0)}_{n-1},0,0,x^{(1)}_2,...,x^{(1)}_{n-1},x^{(1)}_n \rangle\\
&\supseteq \langle x^{(0)}_1,x^{(0)}_2,...,x^{(0)}_{n-1},0,0,0,x^{(1)}_3,...,x^{(1)}_{n-1},x^{(1)}_n \rangle\\
&\supseteq  ... \\
&\supseteq \langle x^{(0)}_1,x^{(0)}_2,...,x^{(0)}_{n-1},0,...,0 \rangle\\
&\supseteq \langle x^{(0)}_1,x^{(0)}_2,...,x^{(0)}_{n-2},0,...,0 \rangle\\
&\supseteq \langle x^{(0)}_1,x^{(0)}_2,...,x^{(0)}_{n-3},0,...,0 \rangle\\
&\supseteq  ... \\
&\supseteq \langle x^{(0)}_1,0,...,0 \rangle\supseteq\{0\}.
\end{split}
\end{equation*}
By direct computations we see $B_{\mathcal{F}}=T_1(I+\pi V)$, where $V$ is the subgroup of $\mathfrak{g}\subseteq M_n(\overline{\mathbb{F}}_q)$ consisting of the matrices $(a_{r,s})_{r,s}$ with $a_{r,s}=0$ for $\forall s\neq n$; in particular, $U_{\mathcal{F}}$ is an $F$-stable subgroup of $B^1=B^{r-1}$, as desired. Note that in both the $r\geq3$ case and the $r=2$ case, we can take $C$ to be $T$, the algebraic group corresponding to the standard maximal torus of $\mathbf{G}$.

\vspace{2mm} This completes the proof of the theorem.
\end{proof}

\begin{remark}
Stasinski communicated to us that, in \cite{Onn_AdvMath_2008} and \cite{Aubert_Onn_Prasad_Stasinski_Israelpaper_2010} a similar but different idea on flags was considered, in which the flags are partial of length one and acted by $\mathcal{O}_r$; note it can happen that an $\overline{\mathbb{F}}_q$-flag in our sense is not ``acted'' by $\mathcal{O}_r$, e.g.\ the space $\langle\pi\rangle=\overline{\mathbb{F}}_q\pi$, viewed as a subspace of the natural $\mathcal{O}_r$-module $\overline{\mathbb{F}}_q[[\pi]]/\pi^r$, does not inherit the action of $\mathcal{O}_r$ on $\overline{\mathbb{F}}_q[[\pi]]/\pi^r$ when $r>2$.
\end{remark}

\section{On regular orbits}\label{sec: reg orbit}

Throughout this section we assume $\mathbb{G}=\mathrm{GL}_n$ with $r \geq 2$, and let $\mathbf{B}$ be the standard upper Borel subgroup and $\mathbf{T}$ the standard maximal torus. Here we focus on the regular orbit representations of $\mathrm{GL}_n(\mathcal{O}_r)$; to construct them via algebraic methods has a long history, and is completed recently in \cite{Stasinski_Stevens_2016_regularRep} and \cite{Krakovski_Onn_Singla_regularchar_2018}.

\vspace{2mm} We recall the works in \cite{Hill_1995_Regular} concerning an analogue of Gelfand--Graev modules. Let $\psi$ be a $\overline{\mathbb{Q}}_{\ell}$-linear character of $U^F$. Let $\mathbf{U}_{i,i+1}$, $i=1,\cdots,n-1$, be the root subgroups for the simple roots. Since $U/[ U,U ]=\prod_{i=1}^{n-1}U_{i,i+1}$, where $U_{i,i+1}$ is the algebraic group corresponding to $\mathbf{U}_{i,i+1}$, we have $\psi=\prod_{i}\psi_i$ where $\psi_i:=\psi|_{U_{i,i+1}^{F}}$. The character $\psi$ is called \emph{non-degenerate}, if $\psi_i$ does not factor through $(U_{i,i+1}^{r-1})^F$ for $\forall i\in\{ 1,\cdots, n-1 \}$. Non-degenerate characters of $U^F$ exist and are all conjugate by $B^F$. 

\vspace{2mm} When $\psi$ is non-degenerate, $\Gamma_{\psi}:=\mathrm{Ind}_{U^F}^{G^F}\psi$ is called a \emph{Gelfand--Graev module} for $\mathrm{GL}_n(\mathcal{O}_r)$. Hill proved the following analogue property of regular characters (see e.g.\ \cite[Chapter~14]{DM1991}):

\begin{prop}[Hill]\label{prop: GG module}
Let $\Gamma_{\psi}$ be a Gelfand--Graev module for $\mathrm{GL}_n(\mathcal{O}_r)$, then every irreducible constituent of $\Gamma_{\psi}$ is regular, and every regular (orbit) representation of $\mathrm{GL}_n(\mathcal{O}_r)$ is an irreducible constituent of $\Gamma_{\psi}$.
\end{prop}
\begin{proof}
See \cite[5.7]{Hill_1995_Regular}.
\end{proof}
Moreover, there is also a multiplicity one property on $\Gamma_{\psi}$ due to Anna Szumovicz (unpublished) and Patel--Singla \cite{Patel_et_Singla_2018_multiplicity-free}.

\vspace{2mm} Using the above proposition we easily deduce that:

\begin{coro}\label{coro: regular orbit}
Every regular orbit representation of $\mathrm{GL}_n(\mathcal{O}_r)$ is an irreducible constituent of $R_{\mathcal{F},C}^{\theta}$ for some admissible triple $(\mathcal{F},C,\theta)$.
\end{coro}
\begin{proof}
Let $\mathcal{F}$ be given by (we follow the notation in the proof of Theorem~\ref{thm: admissible DL rep})
\begin{equation*}
\begin{split}
&\{0\}\subseteq \langle x_1^{(0)} \rangle\subseteq \langle x_1^{(0)}, x_2^{(0)} \rangle\subseteq \cdots \subseteq \langle x_1^{(0)}, x_2^{(0)},\cdots,x_{n}^{(0)}\rangle\\
&\subseteq \langle x_1^{(0)}, x_2^{(0)},\cdots,x_{n}^{(0)}, x_{1}^{(1)}\rangle\subseteq \langle x_1^{(0)}, x_2^{(0)},\cdots,x_{n}^{(0)}, x_{1}^{(1)},x_{2}^{(1)}\rangle\\
&\subseteq\cdots \subseteq \langle x_1^{(0)}, x_2^{(0)},\cdots, x_{n}^{(r-1)}\rangle,
\end{split}
\end{equation*} 
then $B_{\mathcal{F}}=B_1$ and $U_{\mathcal{F}}=U_1$. Note that we can take $C=T_1$. So, similar to the argument of Theorem~\ref{thm: non-nilpotency} we have
\begin{equation}\label{temp: regular flag}
\sum_{\theta\in\widehat{T_1^F}}R_{\mathcal{F},T_1}^{\theta}=\mathrm{Ind}_{U_1^F}^{G^F}1.
\end{equation}
Now let $\psi$ be a non-degenerate character of $U^F$. Consider $\psi':=\widetilde{\psi^{-1}|_{U_1^F}}\cdot\psi$, where $=\widetilde{\psi^{-1}|_{U_1^F}}$ denotes the trivial inflation of $(\psi|_{U_1^F})^{-1}$ along $\rho_{r,1}\colon U^F\rightarrow U_1^F$; in particular, $\psi'$ is a non-degenerate character s.t.\ $\psi'|_{U_1^F}=1$, thus the Frobenius reciprocity
$$\langle {\psi}', \mathrm{Ind}_{U_1^F}^{U^F}1\rangle_{U^F}=\langle 1,1 \rangle_{U_1^F}=1$$
and \eqref{temp: regular flag} imply that $\Gamma_{\psi'}$ is a subrepresentation of $\sum_{\theta\in\widehat{T_1^F}}R_{\mathcal{F},T_1}^{\theta}$. So the assertion follows from Proposition~\ref{prop: GG module}.
\end{proof}

In the remaining part of this section let $r=2$. 

\vspace{2mm} A $\overline{\mathbb{Q}}_{\ell}$-character of the finite abelian group $\mathfrak{g}^F\cong (G^{1})^F$ is called an \emph{invariant character} if it is invariant under the adjoint action of $\mathbb{G}(\mathbb{F}_q)$; an invariant character is called irreducible if it is not the sum of two non-zero invariant characters. This notion has interesting relations with character sheaves; see \cite{Lusztig_1987_fourier_Lie_alg} and \cite{Let2005book}. Letellier conjectured in \cite{Let09CharRedLieAlg} that, for any non-zero irreducible invariant character $\Psi$, there is a higher Deligne--Lusztig representation $\rho$ such that 
$$\langle \Psi, \mathrm{Res}^{G^F}_{(G^{r-1})^F}\rho \rangle_{\mathfrak{g}^F}\neq0;$$ 
this was proved for $\mathrm{GL}_2$ and $\mathrm{GL}_3$ in \cite{ChenStasinski_2016_algebraisation}. Now, from the argument of Corollary~\ref{coro: regular orbit} we see immediately that the conjecture is true for any $\mathrm{GL}_n$ if we allow all the Deligne--Lusztig representations associated to admissible triples: Indeed, from the argument of Corollary~\ref{coro: regular orbit}, there is an $\mathcal{F}$ such that $U_{\mathcal{F}}=U_1$, so by Mackey's intertwining formula and Frobenius reciprocity we have
\begin{equation*}
\begin{split}
\langle \Psi, \mathrm{Res}^{G^F}_{(G^{r-1})^F} \sum_{\theta\in \widehat{T_1^F}} R_{\mathcal{F},T_1}^{\theta} \rangle_{\mathfrak{g}^F}
&=\langle \Psi, \mathrm{Res}^{G^F}_{(G^{r-1})^F} \mathrm{Ind}_{U_1^F}^{G^F} 1 \rangle_{\mathfrak{g}^F}\\
&=\langle \Psi, \sum_{s\in U_1^F\backslash G^F/ (G^{r-1})^F} \mathrm{Ind}_{sU_1^Fs^{-1}\cap (G^{r-1})^F}^{(G^{r-1})^F} \mathrm{Res}_{sU_1^Fs^{-1}\cap (G^{r-1})^F}^{U_1^F} 1 \rangle_{\mathfrak{g}^F}\\
&=\sum_{s\in U_1^F\backslash G^F/ (G^{r-1})^F} \langle \mathrm{Res}_{sU_1^Fs^{-1}\cap (G^{r-1})^F}^{(G^{r-1})^F}\Psi,   \mathrm{Res}_{sU_1^Fs^{-1}\cap (G^{r-1})^F}^{U_1^F} 1 \rangle_{\mathfrak{g}^F}\\
&=\deg\Psi\cdot \#U_1^F\backslash G^F/ (G^{r-1})^F,
\end{split}
\end{equation*}
So $\langle \Psi, \mathrm{Res}^{G^F}_{(G^{r-1})^F}, R_{\mathcal{F},T_1}^{\theta} \rangle_{\mathfrak{g}^F}\neq 0$ for some $\theta$.

\bibliographystyle{alpha}
\bibliography{zchenrefs}

\end{document}